\documentclass[11pt]{amsart}
\usepackage{amsmath,amscd}
\usepackage{amsbsy}
\usepackage{amssymb}
\usepackage{amscd,amsthm}
\usepackage[all,cmtip]{xy}

\newtheorem{thm}{Theorem}\numberwithin{thm}{section}

\newtheorem{cor}[thm]{Corollary}
\newtheorem{exam}[thm]{Example}
\newtheorem{rema}[thm]{Remark}

\newtheorem{con}[thm]{Conjecture}

\newtheorem{defi}[thm]{Definition}

\newtheorem*{thm2}{Theorem}
\newtheorem*{cor2}{Corollary}

\begin{document}
\begin{center}
\huge{No full exceptional collections on non-split Brauer--Severi varieties of dimension $\leq 3$}\\[1cm]
\end{center}
\begin{center}

\large{Sa$\mathrm{\check{s}}$a Novakovi$\mathrm{\acute{c}}$}\\[0,5cm]
{\small March 2016}\\[0,5cm]
\end{center}
{\small \textbf{Abstract}.
In an earlier paper we showed that non-split Brauer--Severi curves do not admit full strong exceptional collections. In the present note we extend this observation and prove that there cannot exist full exceptional collections on non-split Brauer--Severi varieties of dimension $\leq 3$. 
\begin{center}
\tableofcontents
\end{center}
\section{Introduction}
In \cite{NO1} we show that there is no full strong exceptional collection on a Brauer--Severi curve and conjecture that there are no such collections on arbitrary non-split Brauer--Severi varieties (see \cite{NO1}, p.9). The proof of the observation in \cite{NO1} uses the classification of locally free sheaves on a Brauer--Severi curve. For details on this, we refer to \cite{NO} and the references therein. In the present note we give an alternative proof for the fact that there are no full (strong) exceptional collections on Brauer--Severi curves which also works in dimension two and three. Our result is the following:
\begin{thm2}(Theorem 4.5)
Let $k$ be a field. Assume the Braid group $B_{n+1}$ acts transitive on the set of full exceptional collections consisting of arbitrary objects (resp. sheaves or vector bundles) on $\mathbb{P}^n_{\bar{k}}$. Then a $n$-dimensional Brauer--Severi variety $X$ admits a full exceptional collection consisting of arbitrary objects (resp. sheaves or vector bundles) if and only if $X$ is split.
\end{thm2}
\begin{cor2}(Corollary 4.6)
Let $X$ be a Brauer--Severi curve or surface. Then $X$ admits a full exceptional collection consisting of arbitrary objects of if and only if $X$ is split. Moreover, a Brauer--Severi threefold admits a full exceptional collection of vector bundles if and only if it is split.
\end{cor2}
We see that, in particular, a Brauer--Severi variety of dimension $\leq 3$ is rational over $k$ if and only if it admits a full exceptional collection.
The above theorem and corollary might be known to the experts but unfortunately we could not find a proof for them in the literature. The proof we give, relies on the assumption that the Braid group action on the set of full exceptional collections on $\mathbb{P}^n$ is transitive and indeed, this is the case for $n\leq 3$ implying Corollary 4.6. It is not clear to me how to prove that there are no full exceptional collections on non-split Brauer--Severi varieties without using this fact. Knowing this, it may give us a way to find a general argument which also works in dimension $\geq 4$.\\

{\small \textbf{Conventions}. Throughout this note $k$ denotes an arbitrary field and $k^s$ and $\bar{k}$ a separable respectively algebraic closure. Furthermore, by a vector bundle we mean a locally free sheaf of finite rank.

\section{Brauer--Severi varieties}
All we state in this section can be found in \cite{GS}. A $k$-scheme $X$ is called a \emph{Brauer--Severi variety} if $X\otimes_k k^s \simeq \mathbb{P}^n_{k^s}$ for some $n$. We say $X$ is split if $X\simeq \mathbb{P}^n_k$. A field extension $k\subset L$ such that $X\otimes_k L\simeq \mathbb{P}^n_L$ is called splitting field for $X$. It can be shown that there is always a finite Galois extension $k\subset E$ being a splitting field for $X$. Via Galois cohomology, $n$-dimensional Brauer--Severi varieties are in one-to-one correspondence with central simple $k$-algebras of degree $n+1$. Recall, a finite dimensional associative $k$-algebra $A$ is called \emph{central simple} if the only two-sided ideals are $0$ and $A$ and whose center equals $k$. By theorems of Noether, K\"oethe and Wedderburn, there exists a finite separable field extension $k\subset L$ that \emph{splits} $A$, i.e., $A\otimes_k L\simeq M_n(L)$. In particular, the dimension of a central simple $k$ algebra $A$ is a square and we set the \emph{degree} of $A$ to be $\sqrt{\mathrm{dim}_k(A)}$. Two central simple algebras $A$ and $B$ are said to be \emph{Brauer equivalent} if there exist integers $r,s$ such that $M_r(A)\simeq M_s(B)$. 

\section{Exceptional collections}
Let $\mathcal{D}$ be a triangulated category and $\mathcal{C}$ a triangulated subcategory. The subcategory $\mathcal{C}$ is called \emph{thick} if it is closed under isomorphisms and direct summands. For a subset $A$ of objects of $\mathcal{D}$ we denote by $\langle A\rangle$ the smallest full thick subcategory of $\mathcal{D}$ containing the elements of $A$. 
Furthermore, we define $A^{\perp}$ to be the subcategory of $\mathcal{D}$ consisting of all objects $M$ such that $\mathrm{Hom}_{\mathcal{D}}(E[i],M)=0$ for all $i\in \mathbb{Z}$ and all elements $E$ of $A$. We say that $A$ \emph{generates} $\mathcal{D}$ if $A^{\perp}=0$. Now assume $\mathcal{D}$ admits arbitrary direct sums. An object $B$ is called \emph{compact} if $\mathrm{Hom}_{\mathcal{D}}(B,-)$ commutes with direct sums. Denoting by $\mathcal{D}^c$ the subcategory of compact objects we say that $\mathcal{D}$ is \emph{compactly generated} if the objects of $\mathcal{D}^c$ generate $\mathcal{D}$. It is a non-trivial fact that for a compactly generated triangulated category $\mathcal{D}$ a set of objects $A\subset \mathcal{D}^c$ generates $\mathcal{D}$ if and only if $\langle A\rangle=\mathcal{D}^c$ (see \cite{BV}). Note that for smooth projective $k$-schemes $X$, the bounded derived category of coherent sheaves $D^b(X)$ is compactly generated (see \cite{BV}). The next two definitions can be found for instance in \cite{HUY}.
\begin{defi}
\textnormal{Let $X$ be a smooth projective $k$-scheme. An object $\mathcal{E}^{\bullet}\in D^b(X)$ is called \emph{exceptional} if $\mathrm{End}(\mathcal{E}^{\bullet})=k$ and $\mathrm{Hom}(\mathcal{E}^{\bullet},\mathcal{E}^{\bullet}[r])=0$ for $r\neq 0$.} 
\end{defi}
\begin{defi}
\textnormal{A totally ordered set $\{\mathcal{E}^{\bullet}_1,...,\mathcal{E}^{\bullet}_n\}$ of exceptional objects on a smooth projective $k$-scheme $X$ is called an \emph{exceptional collection} if $\mathrm{Hom}(\mathcal{E}^{\bullet}_i,\mathcal{E}^{\bullet}_j[r])=0$ for all integers $r$ whenever $i>j$. An exceptional collection is \emph{full} if $\langle\{\mathcal{E}^{\bullet}_1,...,\mathcal{E}^{\bullet}_n\}\rangle=D^b(X)$ and \emph{strong} if $\mathrm{Hom}(\mathcal{E}^{\bullet}_i,\mathcal{E}^{\bullet}_j[r])=0$ whenever $r\neq 0$.}
\end{defi}
\begin{exam}
\textnormal{Let $X=\mathbb{P}^n_k$ and consider the ordered collection of invertible sheaves $\{\mathcal{O}, \mathcal{O}(1),...,\mathcal{O}(n)\}$. In \cite{BE} Beilinson showed that this is a full strong exceptional collection. }
\end{exam}

An exceptional collection consisting of two exceptional objects is called \emph{exceptional pair}. Now let $\{\mathcal{E}^{\bullet},\mathcal{F}^{\bullet}\}$ be an exceptional pair. There exists a natural map
\begin{eqnarray*}
\mathrm{Hom}(\mathcal{E}^{\bullet},\mathcal{F}^{\bullet})\otimes^{\mathbb{L}}\mathcal{E}^{\bullet}\longrightarrow \mathcal{F}^{\bullet} 
\end{eqnarray*}
given as $\phi\otimes e\mapsto \phi(e)$. Similar for any integer $r\in\mathbb{Z}$ one has
\begin{eqnarray*}
\mathrm{Hom}(\mathcal{E}^{\bullet}[-r],\mathcal{F}^{\bullet})\otimes^{\mathbb{L}}\mathcal{E}^{\bullet}[-r]\longrightarrow \mathcal{F}^{\bullet} 
\end{eqnarray*}
and if we take direct sums we get the \emph{canonical map} and a triangle in $D^b(X)$ given as
\begin{eqnarray*}
L_{\mathcal{E}^{\bullet}}\mathcal{F}^{\bullet}\longrightarrow \bigoplus_r\mathrm{Hom}(\mathcal{E}^{\bullet}[-r],\mathcal{F}^{\bullet})\otimes^{\mathbb{L}}\mathcal{E}^{\bullet}[-r]\longrightarrow \mathcal{F}^{\bullet}\longrightarrow L_{\mathcal{E}^{\bullet}}\mathcal{F}^{\bullet}[1]. 
\end{eqnarray*}
The object $L_{\mathcal{E}^{\bullet}}\mathcal{F}^{\bullet}\in D^b(X)$ is called the left \emph{mutation} of $\mathcal{F}^{\bullet}$ by $\mathcal{E}^{\bullet}$. If we consider the dual of the canonical map, we obtain
\begin{eqnarray*}
\mathcal{E}^{\bullet}\longrightarrow \bigoplus_r\mathrm{Hom}(\mathcal{E}^{\bullet},\mathcal{F}^{\bullet}[r])^*[r]\otimes^{\mathbb{L}}\mathcal{F}^{\bullet}
\end{eqnarray*}
and hence a triangle
\begin{eqnarray*}
\mathcal{E}^{\bullet}\longrightarrow \bigoplus_r\mathrm{Hom}(\mathcal{E}^{\bullet},\mathcal{F}^{\bullet}[r])^*[r]\otimes^{\mathbb{L}}\mathcal{F}^{\bullet}\longrightarrow R_{\mathcal{F}^{\bullet}}\mathcal{E}^{\bullet}\longrightarrow \mathcal{E}^{\bullet}[1].
\end{eqnarray*}
The object $R_{\mathcal{F}^{\bullet}}\mathcal{E}^{\bullet}$ is called the right mutation of $\mathcal{E}^{\bullet}$ by $\mathcal{F}^{\bullet}$. For details and a review of the theory on exceptional collections and mutations we refer to \cite{GK}.

Given an exceptional collection $\{\mathcal{E}^{\bullet}_1,...,\mathcal{E}^{\bullet}_n\}$ we can consider any exceptional pair $\{\mathcal{E}^{\bullet}_i,\mathcal{E}^{\bullet}_{i+1}\}$ and perform left or right mutations to get a new exceptional collection. In this context one can show that Artin's Braid group $B_n$ with $n$ strings naturally acts on the set of full exceptional collections in $D^b(X)$, where $X$ is a smooth projective integral $k$-scheme with $\mathrm{rk}(K_0(X))=n$ via left and right mutations. Bondal \cite{BO} proved that in case $\mathrm{rk}(K_0(X))=\mathrm{dim}(X)+1$, the property of an exceptional collection to consist only of exceptional sheaves is preserved by mutations. One principal problem of the theory of mutations of exceptional objects (resp. sheaves or vector bundles) on $\mathbb{P}^n$ is to show that their action on full exceptional collection consisting of arbitrary objects (resp. sheaves or vector bundles) is transitive. More generally, Bondal and Polishschuk \cite{BP} conjectured that the action of the semidirect product $B_n\ltimes \mathbb{Z}^n$ induced from mutations and shifts on full exceptional collections in any triangulated category $\mathcal{D}$ is transitive.\\ 
A generalization of the notion of a full exceptional collection is that of a semiorthogonal decomposition.
Recall that a full triangulated subcategory $\mathcal{D}$ of $D^b(X)$ is called \emph{admissible} if the inclusion $\mathcal{D}\hookrightarrow D^b(X)$ has a left and right adjoint functor. 
\begin{defi}
\textnormal{Let $X$ be a smooth projective $k$-scheme. A sequence $\mathcal{D}_1,...,\mathcal{D}_n$ of full triangulated subcategories of $D^b(X)$ is called \emph{semiorthogonal} if all $\mathcal{D}_i\subset D^b(X)$ are admissible and $\mathcal{D}_j\subset \mathcal{D}_i^{\perp}=\{\mathcal{F}^{\bullet}\in D^b(X)\mid \mathrm{Hom}(\mathcal{G}^{\bullet},\mathcal{F}^{\bullet})=0$, $\forall$ $ \mathcal{G}^{\bullet}\in\mathcal{D}_i\}$ for $i>j$. Such a sequence defines a \emph{semiorthogonal decomposition} of $D^b(X)$ if the smallest full thick subcategory containing all $\mathcal{D}_i$ equals $D^b(X)$.}
\end{defi}
For a semiorthogonal decomposition we write $D^b(X)=\langle \mathcal{D}_1,...,\mathcal{D}_n\rangle$.
\begin{rema}
\textnormal{Let $\mathcal{E}^{\bullet}_1,...,\mathcal{E}^{\bullet}_n$ be a full exceptional collection on $X$. By setting $\mathcal{D}_i=\langle\mathcal{E}^{\bullet}_i\rangle$ we get a semiorthogonal decomposition $D^b(X)=\langle \mathcal{D}_1,...,\mathcal{D}_n\rangle$ (see \cite{HUY}, Example 1.60).}
\end{rema}
\begin{exam}
\textnormal{Let $X$ be a $n$-dimensional Brauer--Severi variety corresponding to the central simple algebra $A$. Due to a result of Bernardara \cite{BER}, one has a semiorthogonal decomposition $D^b(X)=\langle D^b(k), D^b(A),D^b(A^{\otimes 2}),...,D^b(A^{\otimes n})\rangle$.} 
\end{exam}
For a wonderful and comprehensive overview of the theory on semiorthogonal decompositions in algebraic geometry we refer to \cite{KU}.
\section{No full exceptional collections on Brauer--Severi varieties?} 
Recall the following conjecture, stated in \cite{NO1} and inspired by Proposition 4.8 in loc. cit..
\begin{con}
\textnormal{Let $X\neq\mathbb{P}^n_k$ be a Brauer--Severi variety. Then $X$ does not admit a full (strong) exceptional collection consisting of arbitrary objects.} 
\end{con}
To prove the theorem and the corollary which are stated in the introduction, we need the next two theorems. We start with the following well-known fact. 
\begin{thm}
Let $A$ and $B$ be central simple $k$-algebras. Then $D^b(A)$ and $D^b(B)$ are equivalent if and only if $A$ is Brauer equivalent to $B$.
\end{thm}
\begin{proof}
The assertion is a consequence of the results in \cite{AN}.
\end{proof}
The following theorem is crucial for the proof of Corollary 4.6 below and concerns the Braid group action on full exceptional collections on projective spaces.
\begin{thm}
Let $k=\bar{k}$. The Braid group $B_{n+1}$ acts transitive on the set of full exceptional collections consisting of arbitrary objects on $\mathbb{P}^n_k$ for $n\leq 2$. Moreover, $B_4$ acts transitive of the set of exceptional collection of sheaves resp. vector bundles on $\mathbb{P}_k^3$.
\end{thm}
\begin{proof}
This is well-known and the proof for $n=1$ is easy as $B_2=\mathbb{Z}$. The case $n=2$ was proved by Gorodentsev and Rudakov \cite{GR} and later extended to arbitrary del Pezzo surfaces by Kuleshov and Orlov \cite{KO}. For the case $n=3$ we refer to the work of Nogin \cite{N} and Polishchuk \cite{PO}.
\end{proof}
\begin{rema}
\textnormal{Let $k=\bar{k}$. It is known that exceptional objects in $D^b(\mathbb{P}^1_k)$ are just shifts of coherent exceptional sheaves whereas exceptional objects in $D^b(\mathbb{P}^2_k)$ are even shifts of exceptional vector bundles. For $\mathbb{P}^1_k$ this is a consequence of the fact that any object $\mathcal{E}^{\bullet}$ is of the form  $\bigoplus_{i\in \mathbb{Z}}H^i(\mathcal{E}^{\bullet})[i]$ and for $\mathbb{P}^2_k$ the statement can be found for instance in \cite{KO}, Proposition 2.10. If exceptional objects in $D^b(\mathbb{P}^3_k)$ are also shifts of coherent sheaves or vector bundles is, as far as I know, up to now an open problem.}
\end{rema}
We are now able to prove our main theorem.
\begin{thm}
Let $k$ be a field. Assume the Braid group $B_{n+1}$ acts transitive on the set of full exceptional collections consisting of arbitrary objects (resp. sheaves or vector bundles) on $\mathbb{P}^n_{\bar{k}}$. Then a $n$-dimensional Brauer--Severi variety $X$ admits a full exceptional collection consisting of arbitrary objects (resp. sheaves or vector bundles) if and only if $X$ is split.
\end{thm}
\begin{proof}
If $X$ is split, it is isomorphic to $\mathbb{P}^n_k$ and the existence of a full exceptional collection follows from Beilinson's result (see Example 3.3). We now prove the other implication.

For a $n$-dimensional Brauer--Severi variety $X$ we have $\mathrm{rk}(K_0(X))=n+1$ (see \cite{BLU}). So assuming the existence of a full exceptional collection, we know that the collection is of length $n+1$. So denote by $\mathcal{E}^{\bullet}_1,..., \mathcal{E}^{\bullet}_{n+1}$ the full exceptional collection we assume to exist. It is easy to see that after base change to $\bar{k}$ this exceptional collection remains a full exceptional collection on $\mathbb{P}^n_{\bar{k}}$. From Example 3.6 we know that $X$ has a semiorthogonal decomposition
\begin{eqnarray}
D^b(X)=\langle D^b(k), D^b(A),...,D^b(A^{\otimes n})\rangle,
\end{eqnarray}
where $A$ is the central simple $k$-algebra corresponding to $X$. By construction (see \cite{BER} for details), this semiorthogonal decomposition base changes to the semiorthogonal composition 
\begin{eqnarray}
D^b(\mathbb{P}^n_{\bar{k}})=\langle\mathcal{O},\mathcal{O}(1),...,\mathcal{O}(n+1)\rangle.
\end{eqnarray}
Assuming the existence of the full exceptional collection $\mathcal{E}^{\bullet}_1,..., \mathcal{E}^{\bullet}_{n+1}$ provides us with a semiorthogonal decomposition 
\begin{eqnarray}
D^b(X)=\langle \mathcal{E}^{\bullet}_1,...,\mathcal{E}^{\bullet}_{n+1}\rangle=\langle D^b(k),...,D^b(k)\rangle.
\end{eqnarray}
Note that $D^b(k)$ is equivalent to $\langle\mathcal{E}^{\bullet}_i\rangle$. This equivalence is given by sending the complex $k$, considered as concentrated in degree zero, to the object $\mathcal{E}^{\bullet}$. Note furthermore, that mutation commutes with base change to field extensions, i.e. for an exceptional pair $\{\mathcal{E}^{\bullet},\mathcal{F}^{\bullet}\}$ one has $(R_{\mathcal{F}^{\bullet}}\mathcal{E}^{\bullet})_L=R_{\mathcal{F}^{\bullet}_L}\mathcal{E}^{\bullet}_L$ for any field extension $k\subset L$. The same holds for left mutations. Now by assumption, the Braid group $B_{n+1}$ acts transitive on the set of full exceptional collections on $\mathbb{P}^n_{\bar{k}}$. Since mutation commutes with base change, we conclude that up to mutation the decomposition (3) base changes to a decomposition
\begin{eqnarray}
D^b(\mathbb{P}^n_{\bar{k}})=\langle\mathcal{O}(j),\mathcal{O}(j+1),...,\mathcal{O}(j+n+1)\rangle
\end{eqnarray}
for some $j\in \mathbb{Z}$. Twisting by powers of the canonical sheaf $\mathcal{O}(-n-1)$ and performing mutations again, we can assume $j=0$. This implies that the semiorthogonal decompositions (1) and (3) are equivalent. In particular, this yields $D^b(k)\simeq D^b(A)$. Finally, Theorem 4.2 shows that $A$, and hence $X$, is split. This completes the proof.
\end{proof}

\begin{cor}
Let $X$ be a Brauer--Severi curve or surface. Then $X$ admits a full exceptional collection consisting of arbitrary objects of if and only if $X$ is split. Moreover, a Brauer--Severi threefold admits a full exceptional collection of vector bundles if and only it is split.
\end{cor}
\begin{proof}
This follows from Theorems 4.3 and 4.5.
\end{proof}

\begin{rema}
\textnormal{The proof of Corollary 4.6 uses the fact that any full exceptional collection on $\mathbb{P}^n$ with $n\leq 3$ can be obtained by successive mutations from the collection $\mathcal{O},\mathcal{O}(1),...,\mathcal{O}(n)$ and vice versa. So in order to prove Conjecture 4.1 by using the technique presented in the proof of Theorem 4.5, one must know that the Braid group acts transitive on the set of full exceptional collections on higher dimensional projective spaces. Up to now, this is an open problem.}
\end{rema}

\small{MATHEMATISCHES INSTITUT, HEINRICH--HEINE--UNIVERSIT\"AT 40225 D\"USSELDORF, GERMANY}\\
E-mail adress: novakovic@math.uni-duesseldorf.de

\end{document}